\documentclass[11pt]{amsart}    

\usepackage{graphicx}
\usepackage{color}
 
\usepackage{amsfonts,amsmath,latexsym,amssymb,amsthm}
\usepackage{hyperref}
\usepackage{geometry}
\usepackage{txfonts}

\newtheorem{theorem}{Theorem}
\newtheorem{corollary}{Corollary}

\theoremstyle{remark}
\newtheorem{remark}{Remark}
\newtheorem{example}{Example}

\begin{document}

\title{Spectral geometry in a rotating frame: properties of the ground state}

\author{Diana Barseghyan}  

\address{Department of Mathematics, University of Ostrava,  30. dubna 22, 70103 Ostrava, Czech Republic \and
             Nuclear Physics Institute, Academy of Sciences of the Czech Republic, Hlavn\'{i} 130, 25068 \v{R}e\v{z} near Prague, Czech Republic}
\email{diana.barseghyan@osu.cz}            
 \author{Pavel Exner}
 
 \address{Nuclear Physics Institute, Academy of Sciences of the Czech Republic,
Hlavn\'{i} 130, 25068 \v{R}e\v{z} near Prague, Czech Republic \and 
Doppler Institute, Czech Technical University, B\v{r}ehov\'{a} 7, 11519 Prague, Czech Republic}
  \email{exner@ujf.cas.cz}

\keywords{Rotating quantum system; Dirichlet condition; Ground state eigenvalue; Optimalization; Comparison to a rotating disk}
\subjclass[2010]{35J15; 35P15; 81Q10}

\maketitle

\begin{abstract}
We investigate spectral properties of the operator describing a quantum particle confined to a planar domain $\Omega$ rotating around a fixed point with an angular velocity $\omega$ and demonstrate several properties of its principal eigenvalue $\lambda_1^\omega$. We show that as a function of rotating center position it attains a unique maximum and has no other extrema provided the said position is unrestricted. Furthermore, we show that as a function $\omega$, the eigenvalue attains a maximum at $\omega=0$, unique unless $\Omega$ has a full rotational symmetry. Finally, we present an upper bound to the difference $\lambda_{1,\Omega}^\omega - \lambda_{1,B}^\omega$ where the last named eigenvalue corresponds to a disk of the same area as $\Omega$
\end{abstract}

\section{Introduction}
\label{s:intro}

The subject of this paper are spectral properties of a quantum particle confined to a planar region rotating around a fixed point. To motivate this task, recall that rotation is a frequent instrument to reveal properties of physical systems. Quantum effects associated with rotation attracted a particular attention in connection with properties of ultracold gases such as superfluidity, see \cite[Sec.~VII]{DBZ08}, \cite{RRD11}, or \cite{BCPY08,CPRY13,LS06} and references therein. These effects have many-body nature and their description can be reduced to nonlinear effective theory, but the one-particle Hamiltonians represent an important input in such a description. At the same time, there are situations when the motion of a single atom in a rotating trap is considered \cite{CGV18} as well as other interesting effects associated with rotation \cite{BB96}.

Naturally, the type of confinement is in these problems usually far from a hard-wall `container' being usually modeled by a suitable harmonic potential. Nevertheless, we believe that it is worth to investigate the problem indicated above from several reasons. On the physical side, experimental techniques develop rapidly and it is conceivable that sooner or later the model in question will become closer to experimental reality. At the same time, one can think of classical systems for which this would be a proper description. On the other hand, from the mathematical point of view it is interesting to see how would be results of the spectral geometry modified in the presence of rotation.

Let us describe our problem in more details. We suppose that $\Omega$ is a bounded region, i.e. regular open set in $\mathbb{R}^2$, without loss of generality we may assume that it is \emph{connected} because in the opposite case we can analyze the spectrum referring to each component separately. The main object of our interest is the operator
\begin{equation}\label{H} H_\omega(x_0, y_0)=
-\Delta_D^\Omega+i\omega ((x-x_0)\partial_y-(y-y_0) \partial_x),\;\;\text{with}\;\,\omega>0,\:(x_0, y_0)\in\mathbb{R}^2,
\end{equation}
where $\Delta_D^\Omega$ is the Dirichlet Laplacian on $\Omega$, the Hamiltonian of a particle in such a `container' rotating around the point $(x_0,y_0)$ with the angular velocity $\omega$ which, again without loss of generality, may be supposed to be positive.

Concerning the domain, one may start from the operator defined on $C_0^\infty(\Omega)$ by the right-hand side of \eqref{H}. It is straightforward to check that the associated quadratic form  coincides with the form of the operator
$$
\hat{H}^{\mathrm{in}}_\omega(x_0, y_0)=\left(i \nabla+\frac{\omega}{2} \hat{A}\right)^2- \frac{\omega^2}{4}((x-x_0)^2+(y-y_0)^2), \;\; \hat{A}=(-y+y_0, x-x_0),
$$
defined on the same space. In view of our hypothesis about $\Omega$ the corresponding operator is obviously bounded from below, hence it allows for Friedrichs extension. Its domain coincides with that of first term, and therefore in view of the regularity assumption this extension is
$$
\hat{H}_\omega(x_0, y_0)=\left(i \nabla+\frac{\omega}{2} \hat{A}\right)^2- \frac{\omega^2}{4}((x-x_0)^2+(y-y_0)^2)
$$
with the domain $\mathcal{H}^2(\Omega)\cap \mathcal{H}^1_0(\Omega)$ \cite[Sec.~1.1.2.2]{Ra17} which is also the natural domain of the operator \eqref{H}.

Furthermore, the vector potential $\frac12\omega\hat{A}$ gives rise to a homogeneous `magnetic field', and as a consequence,
$\hat{H}_\omega(x_0, y_0)$ is a by simple gauge transformation, $u(x,y) \mapsto u(x,y) \mathrm{e}^{-i\omega(xy_0-yx_0)/2}$, unitarily equivalent to
\begin{equation}\label{equivalent}
\widetilde{H}_\omega(x_0, y_0)=\left(i \nabla+\frac{\omega}{2} A\right)^2- \frac{\omega^2}{4}((x-x_0)^2+(y-y_0)^2)
\end{equation}
with $A:=(-y, x)$; in what follows we will work with both the operators \eqref{H} and \eqref{equivalent}. The boundedness of the second term at the right-hand side together with the `magnetic' form of the first also implies that the spectrum of $H_\omega(x_0, y_0)$ is purely discrete and the eigenvalues obey for any fixed $\omega>0$ and $(x_0, y_0)\in\mathbb{R}^2$ Weyl's law,
\begin{equation}\label{weyl}
\lambda^\omega_n(x_0, y_0) = \frac{4\pi n}{|\Omega|}\, (1+o(1)) \quad\text{as}\;\; n\to\infty,
\end{equation}
cf.~\cite{FLW09}. Our concern in this paper will be the principal eigenvalue $\lambda^\omega_1(x_0, y_0)$.

We note the operators $\widetilde{H}_\omega (x_0, y_0)$ form an analytic family in the sense of \cite[Chap.~7]{Ka80}. For the first term one can check it, say, as in \cite{EKP16}, the second is even simpler being a multiplication by a function which is analytic in $\omega,\, x_0,\, y_0$ and bounded (on $\Omega$). This means, in particular, that the eigenvalue $\lambda^\omega_1(x_0, y_0)$ is simple for small values of $\omega$ since it has the property for $\omega=0$. On the other hand, we shall see that this may not be the case in general.

Let us briefly describe the contents of the paper. First we will ask about optimization of $\lambda^\omega_1(x_0, y_0)$ with respect to the position of the rotation center. We will show that there is a unique maximum and no other critical points provided $(x_0, y_0)$ is unrestricted. Assuming in addition that $\Omega$ is convex or the rotation is slow, one can also say something about the location of the maximum. Next, in Section~\ref{s:omega-inequality}, we investigate $\lambda^\omega_1(x_0, y_0)$ as a function of the angular velocity, we show that it attains a maximum value when $\omega=0$ and that this maximum is unique unless $\Omega$ exhibits a full rotational symmetry with respect to the point $(x_0, y_0)$.

In Section~\ref{s:disk-comp} we compare for a fixed $\omega$ the principal eigenvalue to that of a disk $B$ of the same area as $\Omega$. In the standstill, $\omega=0$, the answer to the optimization question would be given by Faber-Krahn inequality, and since the first term in the operator \eqref{equivalent} can be written as the magnetic Dirichlet Laplacian, it is worth noting that for such an operator the disk is also an optimizer \cite{Er96}. It is the second term, however, that spoils the game; while for the first one departure from a circular form pushes the energy up, in the second one on the contrary an elongated shape means a larger negative contribution, hence the existence of an optimal shape is not \emph{a priori} clear. What we are able to demonstrate in this situation is an \emph{upper} bound on the difference $\lambda_{1,\Omega}^\omega(x_0, y_0) - \lambda_{1,B}^\omega(0,0)$ for convex regions expressed in terms of the geometry of $\Omega$ and the angular velocity. We conclude the paper by mentioned some open questions inspired by this analysis.

\section{Rotation center optimalization}
\label{optimalization}

The first problem we are going to address concerns the dependence of the ground state energy on the position of center of rotation, i.e. the function $(x_0, y_0) \mapsto \lambda^\omega_1(x_0, y_0)$ for fixed $\Omega$ and $\omega$, in particular, the existence its extrema.

\subsection{Absence of minima and saddle points}
\label{ss:absence}

\begin{theorem}\label{th1}
$\lambda^\omega_1(\cdot, \cdot)$  as a map $\mathbb{R}^2 \to \mathbb{R}$ has no minima. If it has a maximum, it is unique.
\end{theorem}
\begin{proof}
We will use the form (\ref{equivalent}) of the operator. The idea is to look how $\lambda^\omega_1(x_0, y_0)$ behaves with respect to small shifts of the rotation center. We employ normalized eigenfunctions $u^{(x_0, y_0)}_\omega$ and $v^{(x_0, y_0)}_\omega$ corresponding to eigenvalue $\lambda^\omega_1(x_0, y_0)$ such that their perturbations satisfy the relations
\begin{align} \label{perturbation}
\begin{split}
& u^{(x_0+t, y_0)}_\omega=u^{(x_0, y_0)}_\omega+\mathcal{O}(t),
\\
& v^{(x_0, y_0+s)}_\omega=v^{(x_0, y_0)}_\omega+\mathcal{O}(s),
\end{split}
\end{align}
for small values of $t$ and $s$, where the error term is understood in the $L^\infty$ sense. The two functions could be identical but we do not know \emph{a priori} whether the eigenvalue $\lambda^\omega_1(x_0, y_0)$ is simple -- in fact we shall see in Example~\ref{disk-ex} below that this not true in general -- what is important is that such functions can be always found \cite[Sec.~II.5, Thm~3]{Re69} as long as the perturbation is analytic which is the case here as we have already mentioned.

Assume that  at least one of the following conditions holds:
\begin{align} \label{1condition}
\begin{split}
& \int_\Omega (x-x_0) \, |u^{(x_0, y_0)}_\omega|^2\,\mathrm{d}x\,\mathrm{d}y\ne 0,
\\
& \int_\Omega (y-y_0) \, |v^{(x_0, y_0)}_\omega|^2\,\mathrm{d}x\,\mathrm{d}y\ne 0,
\end{split}
\end{align}
then we are going to show that $(x_0, y_0)$ is not an extremum point. Suppose, for instance, that the first integral is positive, then using the minimax principle we get for any $h<0$ small enough
\begin{align}
\nonumber & \lambda^\omega_1(x_0+h, y_0)\le \left(\widetilde{H}_\omega(x_0+h, y_0)u^{(x_0, y_0)}_\omega, u^{(x_0, y_0)}_\omega\right)\\ \nonumber & = \left(\widetilde{H}_\omega(x_0, y_0)u^{(x_0, y_0)}_\omega, u^{(x_0, y_0)}_\omega\right)+ \frac{\omega^2 h}{2} \int_\Omega (x- x_0 ) |u^{(x_0, y_0)}_\omega|^2\,\mathrm{d}x\,\mathrm{d}y+ \mathcal{O}(h^2)\\ \label{max} & <  \left(\widetilde{H}_\omega(x_0, y_0)u^{(x_0, y_0)}_\omega, u^{(x_0, y_0)}_\omega\right)=\lambda^\omega_1(x_0, y_0).
\end{align}
On the other hand, the assumed positivity of the integral in combination with (\ref{perturbation}) means that
$$
\int_\Omega (x-x_0) \, |u^{(x_0+t, y_0)}_\omega|^2\,\mathrm{d}x\,\mathrm{d}y> 0
$$
holds for all sufficiently small $t$. Choosing the latter positive and small enough, we get
\begin{align*}
& \lambda^\omega_1(x_0, y_0)\le \left(\widetilde{H}_\omega(x_0, y_0)u^{(x_0+t, y_0)}_\omega, u^{(x_0+t, y_0)}_\omega\right) \\ & = \left(\widetilde{H}_\omega(x_0+t , y_0 )u^{(x_0+t, y_0)}_\omega, u^{(x_0+t, y_0 )}_\omega\right) \\ & \quad -\frac{\omega^2 t}{2} \int_\Omega (x-x_0- t)\, |u^{(x_0 +t, y_0 )}_\omega|^2\,\mathrm{d}x\,\mathrm{d}y+\mathcal{O}(t^2) \\ & < \left(\widetilde{H}_\omega(x_0+t, y_0 )u^{(x_0 +t, y_0)}_\omega, u^{(x_0+t, y_0)}_\omega\right)=\lambda^\omega_1(x_0 +t, y_0).
\end{align*}
A similar argument applies if the integral is negative, which together with (\ref{max}) shows that $x_0$ is not a stationary point of the function $\lambda^\omega_1(\cdot, y_0)$. Repeating the reasoning with the other variable, we conclude that if $\lambda^\omega_1(\cdot,\cdot)$ has a stationary point $(x_0, y_0)$, there exist eigenfunctions $u^{(x_0, y_0)}_\omega$ and $v^{(x_0, y_0)}_\omega$, possibly equal each other, which satisfy
\begin{align} \label{extr.}
\begin{split}
& \int_\Omega (x-x_0) \, |u^{(x_0, y_0)}_\omega|^2\,\mathrm{d}x\,\mathrm{d}y = 0,
\\
& \int_\Omega (y-y_0) \, |v^{(x_0, y_0)}_\omega|^2\,\mathrm{d}x\,\mathrm{d}y = 0.
\end{split}
\end{align}
In such a case we have for all nonzero and sufficiently small $h$
\begin{align*}
& \lambda^\omega_1(x_0+h, y_0)\le \left(\widetilde{H}_\omega(x_0+h, y_0)u^{(x_0, y_0)}_\omega, u^{(x_0, y_0)}_\omega\right) \\ & =  \left(\widetilde{H}_\omega(x_0, y_0)u^{(x_0, y_0)}_\omega, u^{(x_0, y_0)}_\omega\right)-\frac{\omega^2 h^2}{4}<\lambda^\omega_1(x_0, y_0)
\end{align*}
and an analogous inequality in the other variable which means that $(x_0, y_0)$ is a point of maximum. This proves the first claim of the theorem.

To establish the second one, we note two things. First of all, the above argument shows that a maximum of the function $\lambda^\omega_1(\cdot,\cdot)$ is an isolated point. Furthermore, the previous reasoning can apply not only in parallel with the axes but in any direction. Suppose thus that the function $\lambda^\omega_1(\cdot,\cdot)$ has more than one maximum point. We pick two of them, and since the spectral properties are obviously invariant with respect to simultaneous Euclidean transformations of all the coordinates, we may place those without loss of generality to the points $(0, 0)$ and $(t_0, bt_0)$ for some $t_0\ne 0$ and $b>0$. Since the function $t\mapsto \lambda_1^\omega (t,bt)$ is continuous it must have a point of minimum at some $t_1\in (0,t_0)$.

As before there is an eigenfunction $u^{(t,bt)}_\omega$ corresponding to $\lambda_1^\omega (t,bt)$ such that $u^{(t,bt)} _\omega =u^{(t_1, bt_1)}_\omega+\mathcal{O}(t-t_1)$ holds for small $|t-t_1|$. In a similar way as above we check that if
$$
\int_\Omega \big(x+by-(1+b^2)t_1\big)\, |u^{(t,bt)}_\omega|^2\,\mathrm{d}x\,\mathrm{d}y\ne 0,
$$
the point $(t_1,bt_1)$ cannot be a stationary point of $\lambda_1^\omega (\cdot,b\,\cdot)$, while the integral vanishes we have
\begin{align*}
& \lambda^\omega_1(t,bt)\le \left(\widetilde{H}_\omega(t,bt)u^{(t_1,bt_1)}_\omega, u^{(t_1, bt_1)}_\omega\right) \\ & = \left(\widetilde{H}_\omega(t_1, bt_1)u^{t_1,bt_1)}_\omega, u^{(t_1,bt_1)}_\omega\right)-\frac{\omega^2 h^2}{4}\,(1+b^2) (t-t_1)^2<\lambda^\omega_1(t_1,bt_1),
\end{align*}
which contradicts the assumption that $(t_1,bt_1)$ is a minimum of $\lambda_1^\omega (\cdot,b\,\cdot)$. In this way the proof is concluded. \qed
\end{proof}

In fact, the claim of the theorem can be strengthened in different ways:

\begin{corollary}\label{fall}
$\lambda_1^\omega(x_0, y_0)\to -\infty$ holds as $(x_0, y_0)\to \infty$.
\end{corollary}
\begin{proof}
By minimax principle we have
$$
\lambda_1^\omega(x_0, y_0)\le \int_\Omega \left|i \nabla u +\frac{\omega}{2} A u\right|^2\,\mathrm{d}x\,\mathrm{d}y-
\frac{\omega^2}{4} \int_\Omega ((x-x_0)^2+ (y-y_0)^2) |u|^2\,\mathrm{d}x\,\mathrm{d}y
$$
for a fixed function $u\in \mathcal{H}_0^1(\Omega)\cap \mathcal{H}^2(\Omega)$ with first term independent of $(x_0, y_0)$,
while the negative second diverges as $(x_0, y_0)\to \infty$. \qed \end{proof} 

Furthermore, since the dependence on the rotation center position is continuous, in fact analytic, we can also make the following claim:

\begin{corollary}\label{max-ex}
the function $\lambda^\omega_1(\cdot, \cdot)$ as a map $\mathbb{R}^2 \to \mathbb{R}$ has a unique maximum.
\end{corollary}

\begin{remark}\label{inner}
The problem considered so far in which the placement of the rotation center is not restricted could be called \emph{free}. One can also consider the \emph{inner} problem where one requires that $(x_0,y_0)\in\bar\Omega$. In that case the minimum of $\lambda^\omega_1(\cdot, \cdot)$ exists and follows from Theorem~\ref{th1} that it is reached at the boundary; it may not be unique as the disk example discussed below shows. The maximum can occur either in a point of $\Omega$ in which case it is unique or at $\partial\Omega$ when it again may not be unique. 
\end{remark}

\subsection{Convex sets}
\label{ss:convex sets}

For a particular class of regions $\Omega$ the difference between optimization result for the free and inner problem mentioned above is absent.
\begin{theorem} \label{th:convex}
Let $\Omega$ be convex, then the ground state eigenvalue of $H_\omega(x_0, y_0)$ reaches its maximum at a point $(x_0, y_0)\in\Omega$.
\end{theorem}
\begin{proof}
One can consider the operator in the two unitarily equivalent forms, either \eqref{H} which which can be using the polar coordinates $(r,\varphi)$ written also as
\begin{equation}\label{Hpolar}
H_\omega(x_0, y_0)=-\Delta_D^\Omega +i \omega \partial_\varphi,
\end{equation}
or (\ref{equivalent}); in both of them it is obvious that spectral properties are invariant with respect to translation and rotations of the coordinate system. We know that the maximum of $\lambda^\omega_1(\cdot, \cdot)$ exists and the conditions (\ref{extr.}) applied to the eigenfunctions $u_\omega^{(x_0, y_0)}$ and $v_\omega^{(x_0, y_0)}$ we conclude that
\begin{equation} \label{projection}
x_0\in \mathrm{Pr}_X(\Omega),\quad y_0\in \mathrm{Pr}_Y(\Omega),
\end{equation}
projections of $\Omega$ to the axes. As (an open) convex set $\Omega$ is star-shaped with respect to any of its points, in other words, one can identify such a point with the origin of the coordinates and write $\Omega= \{r\in[0,R(\varphi)),\,\varphi\in[0,2\pi)\}$ for a suitable $2\pi$-periodic function $R$. Assume that $(x_0, y_0)\not\in\Omega$, then one can always rotate the coordinate system in such a way that at least one of the conditions \eqref{projection} with respect to the rotated axes $X',\,Y'$ would be violated -- the point would appear either outside the segment or at its boundary -- which contradicts the mentioned rotational invariance. \qed
\end{proof}

\subsection{Slow rotation}
\label{ss:slow}

If assume in addition that the angular velocity $\omega$ is small one located the position of the maximum more precisely in a way reminiscent of \cite{HKK01,EM08}. To this aim, we need a notation: given a region $\Sigma\subset\mathbb{R}^2$ and a line $P$, we
denote by $\Sigma^P$ the mirror image of $\Sigma$ with respect to $P$.
\begin{theorem}
Let $\Omega$ be convex set and $P$ be a line which divides $\Omega$ into two parts, $\Omega_1$ and $\Omega_2$, in such a way that
$\Omega_1^P\subset \Omega_2$. Then for small enough values of $\omega$ the point at which $\lambda^\omega_1(x_0, y_0)$ attains its maximum does not belong to $\Omega_1$.
\end{theorem}
\begin{proof}
Using once more the invariance of the spectrum with respect to rotations, we may suppose that $P$ is parallel to the $Y$ axis.
Consider first the situation when the rotation is absent, $\omega=0$. Let $u_D$ be the ground state eigenfunction of the Dirichlet Laplacian, $-\Delta^D_\Omega u_D= \lambda_1^D u_D$. We employ the idea introduced in \cite{HKK01} and define the function
$$
v(x, y):=u_D(x, y)- u_D(x^P, y) \quad \text{on}\;\; \Omega_1,
$$
where $(x^P, y)$ is the mirror image of $(x, y)$ with respect to $P$; since $\Omega_1^P\subset \Omega_2$ the function is well defined. The ground state eigenfunction of the Dirichlet Laplacian can be chosen to be positive, hence we have $v|_{\partial \Omega_1}\le 0$. This fact together with the equation $-\Delta v= \lambda_1^D v$ on $\Omega_1$ and the maximum principle for the second order elliptic partial differential equations means that $v< 0$ holds on $\Omega_1$, in other words
\begin{equation}\label{negativeness}
u_D(x, y)\le u_D(x^P, y),\quad (x, y)\in \Omega_1.
\end{equation}
Now we take an arbitrary point $(x_0,y_0)\in\Omega_1$ and assume that $P$ passes through it. Using (\ref{negativeness}) we find
\begin{align}
& \nonumber \int_\Omega (x-x_0) (u_D (x, y))^2\,\mathrm{d}x\,\mathrm{d}y=\int_{\Omega^\backslash\left (\Omega_1\cup\Omega_1^P\right)} (x-x_0) (u_D (x, y))^2\,\mathrm{d}x\,\mathrm{d}y \\ & \nonumber \quad + \int_{\Omega_1}(x-x_0) (u_D(x, y))^2\,\mathrm{d}x\,\mathrm{d}y + \int_{\Omega_1^P} (x-x_0) (u_D(x, y)|)^2\,\mathrm{d}x\,\mathrm{d}y \\ & \label{neg.} > \int_{\Omega^\backslash\left (\Omega_1\cup \Omega_1^P\right)} (x-x_0) (u_D (x, y))^2\,\mathrm{d}x\,\mathrm{d}y > 0.
\end{align}
So far the point $(x_0,y_0)\in\Omega_1$ was arbitrary, now we replace $-\Delta^D_\Omega$ by the operator \eqref{H} and assume it $(x_0,y_0)$ is the point when the maximum is attained. Our operators form a type A analytic family \cite{Ka80} with respect the parameters. This means, in particular, that the ground state eigenvalue $\lambda^\omega_1(x_0, y_0)$ is simple for all sufficiently small $\omega$ and the corresponding eigenfunction satisfies
$$
u^\omega(x_0, y_0)(x, y)=u_D(x, y) + \mathcal{O}(\omega)
$$
with the error understood in the $L^\infty$ sense. In combination with inequality (\ref{neg.}) and Theorem~\ref{th:convex} this leads to a contradiction with the condition (\ref{extr.}) concluding thus the proof. \qed
\end{proof}

\section{Standstill costs energy}
\label{s:omega-inequality}

Our next topic is to compare the ground state eigenvalue of $H_\omega(x_0, y_0)$ in the situation with and without the rotation. Specifically, we are going to prove the following claim:

\begin{theorem}\label{th:omega-inequality}
Under the stated assumptions we have
\begin{equation} \label{omega-inequality}
\lambda^\omega_1(x_0, y_0) \le \lambda_1^D(\Omega),
\end{equation}
where $\lambda_1^D(\Omega)$ is the ground state eigenvalue of the Dirichlet Laplacian $-\Delta_D^\Omega$ on $\Omega$. Moreover, the inequality is sharp for $\omega>0$ provided the region $\Omega$ does not have full rotational symmetry being a disk or a circular annulus rotating around its center.
\end{theorem}
\begin{proof}
We employ the form (\ref{H}) of the Hamiltonian. The inequality \eqref{omega-inequality} is easy to show, by minimax principle we have
$$
\lambda^\omega_1(x_0, y_0) \le (H_\omega(x_0, y_0)u_D,u_D) = \lambda_1^D(\Omega) + (u_D, i\omega ((x-x_0)\partial_y-(y-y_0) \partial_x)u_D)
$$
but the second term on the right-hand side vanishes because the ground state eigenfunction $u_D$ of $-\Delta_D^\Omega$ can be chosen real-valued.

Suppose next that the system lacks the rotational symmetry. First we will show that $\omega=0$ is a strict maximum of the function $\omega\mapsto \lambda_1^\omega(x_0, y_0)$. For small $\omega$ we have the perturbation expansion
$$
\lambda_1^\omega(x_0, y_0)=\lambda_1^D +a \omega+ b \omega^2+\mathcal{O}(\omega^3),
$$
with
\begin{align}
\label{a} & a=i \int_\Omega \big((x-x_0)(u_D)_y- (y-y_0)(u_D)_x\big) \,\overline{u}_D\,\mathrm{d}x\,\mathrm{d}y, \\ \nonumber
& b=\sum_{k=2}^\infty\frac{1}{(\lambda^D_1(\Omega)-\lambda^D_k(\Omega))} \left|\int_\Omega \big((x-x_0)(u_D)_y-(y-y_0)(u_D)_x\big) \,\overline{u}_{D,k}\,\mathrm{d}x\,\mathrm{d}y\right|^2
\end{align}
and the natural partial derivative abbreviations, where $\{u_{D,k}:\,k=1,2,\dots\}$ are the eigenfunctions of the Dirichlet Laplacian $-\Delta_D^\Omega$ corresponding to eigenvalues $\lambda_k^D(\Omega)$ and $u_D:= u_{D,1}$. Using again the fact that $u_D$ can be chosen real-valued we find that $a=0$, thus it remains show that $b<0$. Since $\lambda_k^D(\Omega)>\lambda_1^D(\Omega)$ holds for all $k\ge 2$, this would mean
$$
\int_\Omega \big((x-x_0)(u_D)_y-(y-y_0)(u_D)_x\big) \, \overline{u}_{D,k}\,\mathrm{d}x\,\mathrm{d}y=0.
$$
This together with (\ref{a}) implies that
$$
\int_\Omega \left((x-x_0)(u_D)_y-(y-y_0)(u_D)_x\right)\,v\,\mathrm{d}x\,\mathrm{d}y=0
$$
must hold for any $v\in L^2(\Omega)$, and this in turn is possible if and only if
$$
(x-x_0)(u_D)_y-(y-y_0)(u_D)_x=0.
$$
The solution to this first order partial differential equation is
$$
u_D(x, y)=f((x-x_0)^2+(y-y_0)^2)\quad\;\text{with any} \;\; f\in C^1(\Omega).
$$
Since $\Omega$ is by assumption neither a disk nor a circular annulus with the center at $(x_0, y_0)$, there must be an open subset of $\Omega$ where $u_D$ vanishes but this contradicts the unique continuation principle for second-order elliptic equations. This proves that $b<0$ and the point $\omega=0$ is a strict maximum.

Assume next that function $\omega\mapsto \lambda_1^\omega(x_0, y_0)$ has another maximum at a point $\omega_0> 0$, then it must have a minimum at some $\omega_1\in (0, \omega_0)$. We are going to show that the corresponding ground state eigenfunction $u_{\omega_1}^{(x_0, y_0)}$ which we for brevity denote as $g$ must satisfy
\begin{equation}
\label{condition3}
i\int_\Omega \big((x-x_0) g_y(x_0, y_0)- (y-y_0)g_x(x_0, y_0)\big) \overline{g}\,\mathrm{d}x\,\mathrm{d}y=0.
\end{equation}
If it is not the case, the left-hand side of the last expression is either positive or negative, because as as a matrix element of a self-adjoint operator it is real valued. In the former case we have for any $h<0$
\begin{align}
\nonumber &\lambda^{\omega_1+h}_1(x_0, y_0)\le \big(H_{\omega_1+h}(x_0, y_0 )g,g\big) \\ \nonumber & = \big(H_{\omega_1}(x_0, y_0 )g,g\big) + i h \int_\Omega \big((x-x_0) g_y- (y-y_0)g_x\big) \overline{g}\,\mathrm{d}x\,\mathrm{d}y \\ \label{max4} & < \big(H_{\omega_1}(x_0, y_0)g,g\big) = \lambda_1^{\omega_1}(x_0, y_0),
\end{align}
hence $\omega_1$ cannot be a minimum. The analogous argument applies if the expression in question is negative, which establishes the validity of (\ref{condition3}). Combining it with the fact that $H_0(x_0, y_0)= -\Delta_D^\Omega$ and the minimax principle, we get
$$
\lambda_1^D(\Omega) \le \big(H_0(x_0, y_0 )g,g\big) =  \big(H_{\omega_1}(x_0, y_0)g,g\big)=\lambda_1^{\omega_1}(x_0, y_0).
$$
This together with the fact that $\omega=0$ is a maximum and $\omega_1$ is a minimum implies that $\lambda_1^{\omega_1}(x_0, y_0)=\lambda_1^D(\Omega)$ which means that the function  $\lambda_1^{\omega}(x_0, y_0)$ has to be constant on the interval $(0,\omega_1)$, however, this is not possible because the origin is a strict maximum. \qed
\end{proof}

Let us now look more closely at the situation when the system has a rotational symmetry and the existence of a sharp maximum is not guaranteed by the previous theorem.
\begin{example} \label{disk-ex}
Let $\Omega$ be a disk of radius $R$ rotating around its center which we identify with the point $(0,0)$. The spectrum is in this case easy to find being
\begin{equation} \label{disk-ev}
\lambda_{m,k}(R,\omega) = \frac{j_{m,k}^2}{R^2}-m \omega, \quad m\in\mathbb{Z}, \:k\in \mathbb{N},
\end{equation}
where $j_{m,k}$ is the $k$th positive zero of Bessel function of the first kind $J_m$. Indeed, the symmetry allows for a separation of variables; we employ polar coordinates with the eigenfunction Ansatz $u(r,\varphi) = f(r) \mathrm{e}^{i m\varphi}$ for $0\le r\le R_0$ and $0\le \varphi<2\pi$. The equation $H_\omega u=\lambda u$ then for a fixed $m\in\mathbb{Z}$ gives
\begin{equation}\label{Bessel}
-f''(r)-\frac{1}{r}f'(r)+\left(\frac{m^2}{r^2}-\omega m-\lambda\right) f(r)=0
\end{equation}
and changing the variable to $t=r \sqrt{\lambda+m \omega}$ write the solution as the Bessel function $J_m$. The solution has to belong to $\mathcal{H}^1_0(\Omega)$ and the Dirichlet condition imposed at $r=R$ requires $J_m\left(R_0 \sqrt{\lambda+m \omega}\right)=0$ which yields the spectrum \eqref{disk-ev} corresponding to the eigenfunctions $(r,\varphi) \mapsto J_m\left(\frac{j_{m,k}}{R}r\right) \mathrm{e}^{im\varphi}$.

Let us look what this tells us about $\lambda^\omega_1(x_0, y_0):= \min_{m\in\mathbb{Z},k\in \mathbb{N}} \lambda_{m,k}(R,\omega)$. For $\omega=0$ it is equal to $\lambda_1^D(\Omega) = \frac{j_{0,1}^2}{R^2}$ which also the value of $\lambda_{0,1}(R,\omega)$ for any $\omega$. Since the Bessel functions zeros are conventionally numbered in the ascending order, we have to care about $k=1$ only. For each $m\in \mathbb{N}$ there is positive $\omega_m$ such that for $\omega>\omega_m$ we have $\lambda_{m,k}(R,\omega)<\lambda_1^D(\Omega)$. Furthermore, we know that $j_{m,1}>m$, see e.g. \cite{Br93}, hence $\omega_m > \frac{m}{R^2} - \frac{j_{0,1}^2}{mR^2}$. This means, in particular, that the inequality \eqref{omega-inequality} is indeed not sharp in this case turning into equality for $\omega\in(0,\omega_1)$.

At the same time we have found that the ground state eigenvalue of $H_\omega$ need not be simple in general; in the present case it becomes degenerate for $\omega=\omega_1$ as well at the other values of $\omega$ where the two locally lowest eigenvalue lines $\omega \mapsto \lambda_{m,k}(R,\omega)$ cross.
\end{example}

\begin{example} \label{annulus-ex}
Let now $\Omega$ be a circular annulus of the radii $0<R_0<R_1$ rotating around the center at $(0,0)$. Mimicking the argument of the previous example we find that its spectrum consists of the eigenvalues
\begin{equation} \label{annulus-ev}
\lambda_{m,k}(R_0,R_1,\omega) = \lambda_{m,k}(R_0,R_1,0)-m \omega, \quad m\in\mathbb{Z}, \:k\in \mathbb{N},
\end{equation}
where the first term on right-hand side are Dirichlet eigenvalues of the non-rotating annulus. The latter are obtained as solutions of the spectral condition
$$
\left| \begin{array}{cc} J_m(R_0\sqrt{\lambda}) & Y_m(R_0\sqrt{\lambda}) \\ J_m(R_1\sqrt{\lambda}) & Y_m(R_1\sqrt{\lambda}) \end{array} \right| = 0
$$
and the radial parts of eigenfunctions now combine Bessel functions of the first and the second type. As in the previous case, $\lambda^\omega_1(x_0, y_0)$ remains constant for small enough values of $\omega$.
\end{example}

\begin{remark}
With respect to the topic of the previous section, the ground state maximum for disks and annuli is reached when the rotation respect the symmetry. Was it not the case, the point $\omega=0$ would be by Theorem~\ref{th:omega-inequality} a strict maximum but this is not the case. The annulus thus provides an example of a region where the maximum of $(x_0, y_0) \mapsto \lambda^\omega_1(x_0, y_0)$ lies outside $\Omega$ and the (non-unique) maximum is attained at the inner boundary.
\end{remark}
 \begin{remark}
The operator \eqref{H} has a natural scaling property: changing $(x,y)$ to $\eta(x,y)$ one has to replace the angular frequency $\omega$ by $\eta^{-2}\omega$ to get the whole spectrum scaled by the factor $\eta^{-2}$. In this sense we can rephrase the conclusion of the examples saying that the ground state is independent of $\omega$ from a fixed finite interval provided the radii are large enough.
\end{remark}

\section{Domain comparison}
\label{s:disk-comp}

Let us finally consider comparison of spectra for different regions at a fixed angular velocity $\omega$. As we mentioned in the introduction, an optimization \emph{\`{a} la} Faber and Krahn cannot be expected. We are nevertheless able to demonstrate an estimate which in the ground state eigenvalue is compared to that of a disk of the same area. For this purpose we modify the notation and add the index specifying the region writing $H_{\omega, \Omega}(x_0, y_0)$ and $\lambda_{1, \Omega}^\omega(x_0, y_0)$. We restrict our attention to convex regions with a fixed rotation center $(x_0, y_0)\in \Omega$ which we can as in the proof of Theorem~\ref{th:convex} write as
$$
\Omega = \big\{ (x_0+r\cos\varphi,y_0+r\sin\varphi)\,:\; \varphi\in [0,2\pi),\;r\in [0,R(\varphi)) \big\}
$$
for a suitable $2\pi$-periodic function $R$. Then we have the following result:

\begin{theorem}\label{th:revFK}
Suppose that $\pi R_0^2= |\Omega|$ and denote by $B$ the disk of radius  $R_0$, then
\begin{align}\label{FKreverse}
& \lambda_{1, \Omega}^\omega(x_0, y_0) \\ & \le  \lambda_{1, B}^\omega(0, 0)+\left(\int_0^{2\pi} \left(\frac{R'(\varphi )}{R(\varphi)}\right)^2\,\mathrm{d}\varphi\right)\, \mathrm{sup}_{0\le m\le \frac{R_0^2 \omega +\sqrt{R_0^4 \omega^2 +4j_{0,1}^2}}{2}}\frac{j_{m,1}^2-m^2}{2 \pi R_0^2}. \nonumber
\end{align}
\end{theorem}
\begin{proof}
It is obvious that without loss of generality we may put $x_0=y_0=0$. The quadratic form corresponding to $H_{\omega, \Omega}(0, 0)$ is
\begin{equation}\label{form}
Q(H_{\omega, \Omega}(0, 0))(u)= \int_0^{2\pi} \int_0^{R(\varphi)} r \left(|u_r|^2+\frac{1}{r^2}|u_\varphi|^2+i \omega u_\varphi \overline{u}\right)\,\mathrm{d}r\,\mathrm{d}\varphi.
\end{equation}
We are going to pass to new coordinates analogous to those used in \cite{AFK16}) changing the radial one to $r=t R(\varphi)$ with $t\in[0,1]$. This allows to rewrite the form (\ref{form}) as
\begin{align}\nonumber
Q(H_{\omega, \Omega}(0, 0))(u) &= \int_0^{2\pi} \int_0^1 R^2(\varphi) t \biggl(\frac{1}{R^2(\varphi)}|v_t|^2+\frac{1}{R^2(\varphi) t^2}\left|v_\varphi-\frac{R'(\varphi) t}{R(\varphi)} v_t\right|^2 \\ \label{new form} & \quad +i\omega \left(v_\varphi- \frac{R'(\varphi ) t}{R(\varphi)} v_t\right) \overline{v}\biggr)\,\mathrm{d}t\,\mathrm{d}\varphi,
\end{align}
where $v(t,\varphi):=u(R(\varphi)t, \varphi)$.

According to Example~\ref{disk-ex} the the eigenfunctions of $H_{\omega,B}(0,0)$ are obtained through solution of equation \eqref{Bessel} with $\lambda=\lambda_{1,B}^\omega(0,0)$ and $m\in\mathbb{Z}$, in particular, the normalized ground state eigenfunction, looks in the new coordinates like
$$
v(t,\varphi) = \left(2\pi \int_0^{R_0} z f^2(z)\,\mathrm{d}z\right)^{-1/2}f(R_0 t)\,\mathrm{e}^{im\varphi}
$$
for some $m$, where it is obviously sufficient to consider $m\in\mathbb{N}\cup\{0\}$. Substituting it into \eqref{new form}, we get from the minimax principle
\begin{align}\nonumber
& \lambda_{1, \Omega}^\omega(0, 0) \le \left(2\pi \, \int_0^{R_0} z f^2(z)\,\mathrm{d}z\right)^{-1}\int_0^{2\pi} \int_0^1 R^2(\varphi) t \biggl(\frac{R_0^2}{R^2(\varphi)}(f'(R_0 t))^2 \\ \nonumber & \qquad +\frac{1}{R^2(\varphi) t^2}\left|i m f(R_0 t)-\frac{R'(\varphi ) t R_0}{R(\varphi)} f'(R_0 t)\right|^2 \\ \nonumber & \qquad +i \omega \left(i m f(R_0 t)- \frac{R_0 R'(\varphi ) t}{R(\varphi)} f'(R_0 t)\right) f(R_0 t)\biggr)\,\mathrm{d}t\,\mathrm{d}\varphi \\ \nonumber &
= \left(2\pi \, \int_0^{R_0} z f^2(z)\,\mathrm{d}z\right)^{-1}\int_0^{2\pi} \int_0^1 \biggl( t R_0^2 (f'(R_0 t))^2+ \frac{m^2}{t} f^2(R_0 t) \\ \label{form1} & \qquad +\left(\frac{R'(\varphi )}{R(\varphi)}\right)^2 t R_0^2 (f'(R_0 t))^2- \omega  m t R^2(\varphi) f^2(R_0 t)\biggr)\,\mathrm{d}t\,\mathrm{d}\varphi,
\end{align}
because the integral of the last term in the middle expression vanishes. Returning to the original radial variable, it is easy to see that right-hand side of the last estimate coincides with
$$
\lambda_{1, B}^\omega(0, 0)+\frac{1}{2\pi}\int_0^{2\pi} \left(\frac{R'(\varphi )}{R(\varphi)}\right)^2\,\mathrm{d}\varphi\, \, \int_0^{R_0} z (f'(z))^2\,\mathrm{d}z\,\left(\int_0^{R_0} z f^2(z)\,\mathrm{d}z\right)^{-1}.
$$
Next we recall that $f$ satisfies equation \eqref{Bessel}. Multiplying it by $rf$, integrating by parts over $[0,R_0]$, and using the fact that $\lambda_{1,B}^\omega(0,0)$ is given by \eqref{disk-ev} we get
\begin{equation}\label{end}
\int_0^{R_0} r (f'(r))^2\,\mathrm{d}r= \int_0^{R_0} r \left(\frac{(j_{m,1}^2}{R_0^2}- \frac{m^2}{r^2}\right) f^2(r)\,\mathrm{d}r \le \frac{j_{m,1}^2- m^2}{R_0^2}
\end{equation}
if we choose the function properly normalized, $\|f\|_{L^2((0,R_0), r\mathrm{d}r)}=1$. It remains to estimate the fraction at the right-hand side; following Example~\ref{disk-ex} we have to care about those values of $m$ for which
$$
\frac{(j_m^1)^2}{R_0^2}- m \omega\le \frac{(j_0^1)^2}{R_0^2}
$$
holds. Since $j_{m,1}>m$ as mentioned there we conclude that $m$ should not exceed $\frac12\big(R_0^2 \omega +\sqrt{R_0^4 \omega^2+4j_{0,1}^2}\big)$.
This together with (\ref{form1}) and (\ref{end}) establishes the claim of the theorem. \qed
\end{proof}

Let us make a couple of comments on the obtained bound. It is clear that the inequality is saturated if and only if $\Omega$ is a disk. Asking how tight the bound is generally we note that the right-hand side of (\ref{FKreverse}) behaves for large values of $\omega$ as
\begin{equation}\label{behavior*}
\lambda_{1, B}^\omega(0, 0)+ \mathcal{O}\big(R_0^{2/3}\omega^{4/3}\,\big).
\end{equation}
Indeed, the behavior of $j_{m,k}$ for a fixed $k$ and large $m$ is known \cite{Tr49}, in particular, for $k=1$ we have $j_m^1= m+ \mathcal{O}(m^{1/3})$, see \cite[Sec.~9.5.14]{AS72}, hence for large values of $m$ not exceeding  $\frac12\big(R_0^2 \omega +\sqrt{R_0^4 \omega^2+4(j_{0,1}^2}\big)$ one has
$$
\frac{(j_m^1)^2- m^2}{R_0^2}=\mathcal{O}\big(R_0^{2/3}\omega^{4/3}\,\big),
$$
which establishes (\ref{behavior*}). Furthermore, \eqref{FKreverse} also implies that
\begin{equation}\label{spinning}
\lambda_{1, \Omega}^\omega(x_0, y_0)\to -\infty \quad\;\text{as}\;\; \omega\to \infty.
\end{equation}
To see that it is enough to show that the right-hand side of \eqref{FKreverse} has the same property. Using the result of Example~\ref{disk-ex} in combination with minimax principle and the above asymptotic behavior, we get
$$
\lambda_{1,B}^\omega(0,0) \le \frac{j_{\left[\frac12 R_0^2 \omega\right],1}^2}{R_0^2}- \left[\frac{1}{2}R_0^2 \omega\right]\omega \le
-\frac{R_0^2 \omega^2}{4}\big(1+o(1)\big)+ \mathcal{O}(\omega^{4/3}),
$$
which proves our claim.

\section{Concluding remarks}
\label{s:concl}

The results derived in the previous sections leave various questions open. Let us mention some of them:
\begin{itemize}
\item  concerning the inner problem mentioned in Remark~\ref{inner}, under which conditions is an extremum of $\lambda_1^\omega$ occurring at a boundary of the region $\Omega$, non-convex in case of a maximum, unique?
\item  how does the right-hand side of \eqref{FKreverse} behave with respect to the position of the rotation center? The polar coordinate parametrization is possible with respect to any $(x_0, y_0)\in \Omega$ and the integral in the second term depends on this point continuously, hence there a -- possibly non-unique -- point where it attains its minimum value. What is such a point for a non-circular $\Omega$ and what is its relation to the center of mass of the region?
\item   is there a nontrivial class of $\Omega$ of a fixed area for which a maximizer of $\lambda_1^\omega$ exists, and in such a case, how it looks like?
\end{itemize}
The ground state is not the only spectral information about the operator \eqref{H}:   
\begin{itemize}
\item another point of interest is the distribution of the eigenvalues. It is clear that the Weyl asymptotics \eqref{weyl} cannot be strengthened to a inequality valid for all $n$ in the spirit of P\'olya's conjecture; the latter is known to be violated for magnetic Dirichlet Laplacians \cite{FLW09} and since the second term on the right-hand side of \eqref{equivalent} is negative, the same would be true in our case
\item  instead one may try to find the asymptotics beyond the leading term. In the non-rotating case the second term of the counting function is determined by the length of the boundary,
      $$
     N(\lambda) = \frac{\lambda}{4\pi}\,|\Omega| - \frac{\sqrt{\lambda}}{8\pi}\,|\partial\Omega| + o(\sqrt{\lambda}),
    $$
    provided the boundary is smooth \cite{Iv80}; the question is for which class of $\Omega$ this remains true in the presence of rotation
\item no less interesting are properties of the eigenfunctions. Although they are complex in general, they may have nodal lines. In Example~\ref{disk-ex}, for instance, the eigenfunction $(r,\varphi) \mapsto J_m\left(\frac{j_{m,k}}{R}r\right) \mathrm{e}^{im\varphi}$ vanishes at the circles of radii $r_i = \frac{j_{m,i}}{j_{m,k}}R\,,\: k=1,\dots,k-1$. However, typically the eigenfunctions are expected to vanish at isolated points with the characteristic vortex behavior in their vicinity; one is interested in their number and distribution
\end{itemize}
Possible generalizations of our present problem include
\begin{itemize}
\item  the case when the Hamiltonian contains an additional potential and/or the condition at the boundary of $\Omega$ are altered. Of a particular interest are strongly attractive Robin conditions which make the eigenfunctions localized in the vicinity of the boundary
\item the analogous problem in higher dimensions when the point $(x_0,y_0)$ is replaced an axis around which the corresponding $\Omega$ rotates
\end{itemize}
We prefer to stop here being convinced that the reader can ask other questions inspired by the problem discussed in this paper.

\section*{Acknowledgements}

The work is supported by the Czech Science Foundation (GA\v{C}R) within the project 17-01706S and by the EU project CZ.02.1.01/0.0/0.0/16\textunderscore 019/0000778.  D.B. thanks A.Khrabustovskyi and V.Lotoreichik for useful discussions, P.E. for the hospitality in Institut Mittag-Leffler where a part of this work was done.

\end{document}